\documentclass{amsart}%
\usepackage{amsmath}
\usepackage{amssymb}
\usepackage{amsfonts}
\usepackage{graphicx}%
\providecommand{\U}[1]{\protect\rule{.1in}{.1in}}
\newtheorem{theorem}{Theorem}
\theoremstyle{plain}

\newtheorem{corollary}{Corollary}

\newtheorem{definition}{Definition}
\newtheorem{example}{Example}

\newtheorem{lemma}{Lemma}

\newtheorem{proposition}{Proposition}
\newtheorem{remark}{Remark}

\numberwithin{equation}{section}
\begin{document}
\title[ ]{The Abel-Steffensen inequality in higher dimensions}
\author{Constantin P. Niculescu}
\address{The Academy of Romanian Scientists, Splaiul Independentei No. 54, Bucharest,
RO-050094, Romania.}
\curraddr{Department of Mathematics, University of Craiova, Craiova 200585, Romania}
\email{cpniculescu@gmail.com}
\date{}
\subjclass[2000]{Primary 26D15; Secondary 26B25, 91B30}
\keywords{Abel-Steffensen inequality, absolutely continuous function, convex function,
2d-increasing function, copula, Riemann-Stieltjes integral.}

\begin{abstract}
The Abel-Steffensen inequality is extended to the context of several
variables. Applications to Fourier analysis of several variables and
Riemann-Stieltjes integration are also included.

\end{abstract}
\maketitle

\section{Introduction}

Abel's partial summation formula is a polynomial identity that asserts that
every pair of families $(a_{k})_{k=1}^{n}$ and $(b_{k})_{k=1}^{n}$ of complex
numbers verifies the
\[
\sum_{k=1}^{n}a_{k}b_{k}=\sum_{k=1}^{n-1}\left[  (a_{k}-a_{k+1})\left(
\sum_{j=1}^{k}b_{j}\right)  \right]  +a_{n}\left(  \sum_{j=1}^{n}b_{j}\right)
\]
An immediate consequence is the following inequality, known as Abel's
inequality:
\begin{subequations}
\[
\left\vert \sum_{k=1}^{n}a_{k}b_{k}\right\vert \leq a_{1}\max_{1\leq m\leq
n}\left\vert \sum_{k=1}^{m}b_{k}\right\vert .
\]
whenever $a_{1}\geq a_{2}\geq\cdots\geq a_{n}\geq0$ and $b_{1},b_{2}%
,...,b_{n}\in\mathbb{C}.$ It is this inequality that allowed Abel
\cite{Ab1826} to prove his well known test of convergence for signed series.
See Choudary and Niculescu \cite{CN2014}, for a complete account concerning
the contribution of Abel to this matter.

One hundred years later, Steffensen \cite{St1919} noticed another useful
consequence of Abel's partial summation formula, that will be referred to as
the Abel-Steffensen inequality: if
\end{subequations}
\[
a_{1}\geq a_{2}\geq\cdots\geq a_{n}\geq0\text{ and }\sum_{k=1}^{j}b_{k}%
\geq0\text{ for all }j\in\{1,2,...,n\},\text{ }%
\]
then
\[
\sum_{k=1}^{n}a_{k}b_{k}\geq0.
\]

Using this inequality, Steffensen succeeded to extend Jensen's inequality for
convex functions beyond the framework of positive measures.

As was noticed by Hardy \cite{H1904}, Abel's partial summation formula can be
easily extended to the case of double series as follows:%
\begin{align}
\sum_{i=1}^{p}\sum_{j=1}^{q}a_{ij}u_{ij}  &  =\sum_{i=1}^{p-1}\sum_{j=1}%
^{q-1}\Delta_{ij}\left(  \sum_{k=1}^{i}\sum_{l=1}^{j}u_{kl}\right)
+\sum_{i=1}^{p-1}\Delta_{iq}\left(  \sum_{k=1}^{i}\sum_{l=1}^{q}u_{kl}\right)
\tag{$H$}\label{HA}\\
&  +\sum_{j=1}^{q-1}\Delta_{pj}\left(  \sum_{k=1}^{p}\sum_{l=1}^{j}%
u_{kl}\right)  +a_{pq}\left(  \sum_{k=1}^{p}\sum_{l=1}^{q}u_{kl}\right)
,\nonumber
\end{align}
where%
\[
\Delta_{ij}=\left\{
\begin{array}
[c]{cl}%
a_{i,j}-a_{i+1,j}-a_{i,j+1}+a_{i+1,j+1} & \text{if }1\leq i<p,\text{ }1\leq
j<q\\
a_{i,q}-a_{i+1,q} & \text{if }1\leq i<p,\text{ }j=q\\
a_{p,j}-a_{p,j+1} & \text{if }i=p,\text{ }1\leq j<q.
\end{array}
\right.
\]

An immediate consequence is the Abel-Steffensen inequality for double sums:

\begin{theorem}
\label{thmAbelineqdis}If $\mathbf{a}=(a_{ij})_{i,j}$ and $\mathbf{u}%
=(u_{ij})_{i,j}$ are two double sequences of real numbers such that%
\[
a_{ij}\geq0,\text{ }\Delta_{ij}\geq0\text{ and }\sum_{k=1}^{i}\sum_{l=1}%
^{j}u_{kl}\geq0
\]
for all $i\in\{1,2,...,p\}$ and $j\in\{1,2,...,q\}$, then%
\[
\sum_{i=1}^{p}\sum_{j=1}^{q}a_{ij}u_{ij}\geq0.
\]

\end{theorem}

The property $\Delta_{ij}\geq0$ of the double sequence $\mathbf{a}%
=(a_{ij})_{i,j}$ represents a 2-dimensional analogue of the usual condition of
downward monotonicity for real sequences. In what follows we will refer to it
as the property of $2d$-monotone decreasing.

Hardy used the formula ($H$) to extend the Abel-Dirichlet criterion of
convergence to the case of multiple series such as%
\[
\sum_{m,n\geq1}\frac{\sin\left(  mx+ny\right)  }{\sqrt{m+n}}.
\]
His argument (inspired by the 1-dimensional case) combined the boundedness of
the partial sums of the series $\sum_{m,n\geq1}\sin\left(  mx+ny\right)  $
with the fact that the double sequence $(\left(  m+n\right)  ^{-1/2}%
)_{_{m,n\geq1}}$ is $2d$-monotone decreasing (see Lemma \ref{lem2d} below).

In 1986, David and Jonathan Borwein \cite{BB1986} sketched the necessary
formalism for defining the general concept of $Nd$-monotonicity (for sequences
$f:\mathbb{N}^{N}\rightarrow\mathbb{R})$ and established the extension of
Leibniz test of convergence to the framework of alternating multiple series.
Their work was motivated by the case of a chemical lattice sum,
\[
\sum_{m,n,p\geq1}\frac{(-1)^{m+n+p}}{\sqrt{m^{2}+n^{2}+p^{2}}},
\]
representing the so called Madelung's constant for sodium chloride.

The aim of this paper is to prove an integral analogue of Theorem
\ref{thmAbelineqdis} in the setting of $2d$-monotone functions and to outline
several consequences of it to Fourier analysis and Riemann-Stieltjes integral
of several variables. See Theorem \ref{thmAbelintegral}, Section 3. For
reader's convenience we summarized in Section 2 the main features of
2d-monotonicity, the natural analogue of monotonicity for functions of two
real variables.

\section{$2d$-monotone functions}

Let $A=I\times J$ be a rectangle in $\mathbb{R}^{2}$ (whose sides $I$ and $J$
are intervals parallel to the coordinate axes).

\begin{definition}
A function $f:I\times J\rightarrow\mathbb{R}$ is called $2d$-\emph{monotone}
if the $f$-measure of every compact subinterval $A=\left[  a,b\right]
\times\left[  c,d\right]  \subset I\times J$ is nonnegative, that is,
\begin{equation}
\lbrack f;A]=f(a,c)-f(a,d)-f(b,c)+f(b,d)\geq0. \label{2-incr}%
\end{equation}

The function $f$ is called $2d$-\emph{monotone increasing/decreasing }if $f$
is $2d$-monotone and also\emph{\ }increasing/decreasing in each variable (when
the other is kept fixed).

The function $f$ is called $2d$-\emph{alternating} if $-f$ is $2d$-monotone.
\end{definition}

The terminology introduced by Definition 1 is motivated by the fact that any
$2d$-monotone function\emph{\ }$f$ verifies an inequality of the type
\begin{equation}
\lbrack f;A]\leq\lbrack f;B] \label{2-incr'}%
\end{equation}
for all compact subintervals $A,B\subset I\times J,$ with $A\subset B.$

\begin{remark}
One can easily show that $f$ is $2d$-monotone if and only if for every
interval $A=\left[  a,b\right]  \times\left[  c,d\right]  \subset I\times J$
the function $\Delta_{c}^{d}f(x,y)=:f(x,d)-f(x,c)$ is increasing on $[a,b]$
and the function $\Delta_{a}^{b}f(x,y)=f(b,y)-f(a,y)$ is increasing on
$[c,d].$ Notice that%
\[
\lbrack f;A]=\Delta_{c}^{d}\Delta_{a}^{b}f=\Delta_{a}^{b}\Delta_{c}^{d}f.
\]

\end{remark}

\begin{example}
\label{ex1}The product $f(x)g(y)$ of any pair of increasing/decreasing
functions $f:I\rightarrow\mathbb{R}$ and $g:J\rightarrow\mathbb{R}$ is
$2d$-monotone increasing on $I\times J$; if the functions have opposite
monotonicity, then their product is $2d$-alternating. Finite linear
combinations $\sum_{k}c_{k}f_{k}(x)g_{k}(y),$ \emph{(}with positive
coefficients\emph{)} of such functions have the same nature.
\end{example}

Calculus offers the following useful criterion of $2d$-monotonicity:

\begin{lemma}
\label{lem2d}Assume that $f$ is a continuously differentiable function on a
rectangle $I\times J$ which admits a continuous second order partial
derivative $\frac{\partial^{2}f}{\partial x\partial y}.$ Then $f$ is
$2d$-monotone if and only if%
\begin{equation}
\frac{\partial^{2}f}{\partial x\partial y}\geq0\text{.} \label{diff2-incr}%
\end{equation}

\end{lemma}

Under the assumptions of Lemma \ref{lem2d}, one can prove that the partial
derivative$\frac{\partial^{2}f}{\partial y\partial x}$ also exists and equals
$\frac{\partial^{2}f}{\partial x\partial y}.$ See \cite{AM}.

\begin{proof}
The implication (\ref{diff2-incr}) $\Rightarrow$ (\ref{2-incr}) follows by
integration and a remark due to Aksoy and Martelli \cite{AM} that asserts that
the continuous differentiability of $f$ together with the existence of a
continuous mixed derivative $\frac{\partial^{2}f}{\partial x\partial y}$ imply
the existence of the other mixed derivative and also their equality. Indeed,
for every compact subinterval $A=\left[  a,b\right]  \times\left[  c,d\right]
\subset I\times J$ we have%
\begin{align*}
0  &  \leq\int_{a}^{b}\int_{c}^{d}\frac{\partial^{2}f}{\partial x\partial
y}dydx\\
&  \leq\int_{a}^{b}\left[  \frac{\partial f}{\partial x}(x,d)-\frac{\partial
f}{\partial x}(x,c)\right]  dt=f(b,d)-f(a,d)-f(b,c)+f(a,c).
\end{align*}

Conversely, if the condition (\ref{2-incr}) is fulfilled, then by the
mean-value theorem we get%
\[
\frac{\partial^{2}f}{\partial x\partial y}(u,v)=\lim_{h\rightarrow0}%
\frac{f(u,v)-f(u,v+h)-f(u+h,v)+f(u+h,v+h)}{h^{2}}\geq0.
\]

\end{proof}

According to Lemma 1, the following functions are $2d$-monotone,
\begin{align*}
E(x,y)  &  =x^{2}+y^{2}\text{\quad}(\text{on }\mathbb{R}^{2}),\\
C(x,y)  &  =(e^{x}+e^{y}-1)^{-1}\text{\quad}(\text{on }\mathbb{R}_{+}^{2})\\
\Pi(x,y)  &  =xy\text{\quad}(\text{on }\mathbb{R}^{2}),
\end{align*}
while $h(x,y)=\log(x^{n}+y^{n})$ is $2d$-alternating on $(0,\infty
)\times(0,\infty)$ whenever $n\geq1.$

The convex functions give rise to $2d$-monotone functions in a natural manner.

\begin{example}
\label{ex2}Since any convex function $F:\mathbb{R\rightarrow R}$ verifies the
inequality
\[
(z-y)F(x)-(z-x)F(y)+(y-x)F(z)\geq0
\]
for all $x<y<z$ $($see \emph{\cite{NP2006}}, Lemma \emph{1.3.2}$),$ one can
attach to it a $2d$-monotone function on $\mathbb{R\times R}$ via the formula
$f(x,y)=-F(x-y).$

In a similar way, if $F:[0,\infty)\rightarrow\mathbb{R}$ is a convex function
and $\lambda>0$, then $f(x,y)=F(\lambda\left(  x+y\right)  )$ is a
$2d$-monotone function on $[0,\infty)\times\lbrack0,\infty).$
\end{example}

\begin{example}
\label{ex3}If $F:\mathbb{R}\rightarrow\mathbb{R}$ is a convex function, then
$f(x,y)=\frac{F(x)+F(y)}{2}-F\left(  \frac{x+y}{2}\right)  $ is $2d$%
-alternating on $\mathbb{R\times R}.$
\end{example}

Probability and statistics provide yet another major source of 2d-monotone
functions: the functions that couple multivariate distribution functions to
their one-dimensional marginal distribution functions, called copulas by Sklar.

\begin{definition}
\emph{(Hoeffding, Fr\'{e}chet and Sklar \cite{Sk1997})} A copula is a function
$C\left(  x,y\right)  :[0,1]\times\lbrack0,1]\rightarrow\lbrack0,1]$ that
plays the following two properties:

$(a)$ $($Boundary Conditions$)$ $C(x,0)=C(0,y)=0$ and $C(x,1)=x$, $C(1,y)=y$
for all $x,y\in\lbrack0,1];$

$(b)$ $(2d$-Monotonicity$)$ If $0\leq x_{1}\leq x_{2}\leq1$ and $0\leq
y_{1}\leq y_{2}\leq1$, then%
\[
C\left(  x_{2},y_{2}\right)  -C\left(  x_{1},y_{2}\right)  -C\left(
x_{2},y_{1}\right)  +C\left(  x_{1},y_{1}\right)  \geq0.
\]

\end{definition}

An important class of copulas is that of \emph{Archimedean copulas.} They are
constructed through a continuous, strictly decreasing and convex generator
$\varphi$ as%
\[
A_{\varphi}(u,v)=\varphi^{-1}\left(  \varphi(x)+\varphi(y)\right)  .
\]
Archimedean copulas represent examples of $2d$-monotone increasing functions.

It is worth noticing that copulas are in a one-to-one correspondence with
Markov operators $T:L^{\infty}\left(  [0,1]\right)  \rightarrow L^{\infty
}\left(  [0,1]\right)  $ (as well as with doubly stochastic measures on the
unit square $[0,1]\times\lbrack0,1]).$ Details can found in the book of Nelson
\cite{Ne2013}.

\section{The extension of Abel-Steffensen inequality}

The main feature of a positive measure is that the integral of a nonnegative
function is a nonnegative number. However, this property still works for some
signed measures when restricted to suitable subcones of the cone of positive
integrable functions. This phenomenon, first illustrated by the
Abel-Steffensen inequality in dimension 1, received a great deal of attention
during the last two decades. See \cite{NS2017} and references therein. The
following result provides an integral analogue of Abel-Steffensen inequality
in dimension 2. Its argument can be extended in an evident manner to cover all
dimensions, but details are too technical to be included here.

\begin{theorem}
\label{thmAbelintegral}Suppose that $f$ and $w$ are two real-valued functions
defined on $[a,b]\times\lbrack c,d]$ that fulfill the following two conditions:

$(i)$ $f$ is continuously differentiable and $2d$-monotone decreasing$;$

$(ii)$ $w$ is integrable and $W(x,y)=\int_{a}^{x}\int_{c}^{y}w(s,t)dtds\geq0$
for all $\left(  x,y\right)  \in\lbrack a,b]\times\lbrack c,d].$

Then%
\[
\int_{a}^{b}\int_{c}^{d}f(x,y)w(x,y)dydx\geq f(b,d)W(b,d).
\]

\end{theorem}

\begin{proof}
Using the technique of Dirac sequences of approximating continuous and
integrable functions $f$ by infinitely differentiable functions we may
restrict ourselves to the case where the hypothesis $(i)$ is replaced by the
following stronger condition:%
\[
(i^{\prime})\quad\left\{
\begin{array}
[c]{c}%
f\text{ is nonnegative and twice continuously differentiable}\\
\frac{\partial f(x,d)}{\partial x}\leq0,\text{ }\frac{\partial f(b,y)}%
{\partial y}\leq0\text{ and }\frac{\partial^{2}f(x,y)}{\partial x\partial
y}\geq0
\end{array}
\right.
\]
for all $\left(  x,y\right)  \in\lbrack a,b]\times\lbrack c,d]$. This can be
done by replacing the initial function $f$ with a convolution product
$f_{n}=\rho_{n}\ast f,$ where%
\begin{align*}
\rho_{n}(x,y)  &  =n^{2}\rho(nx,ny)\text{ for }n\geq1,\\
\rho(x,y)  &  =\left\{
\begin{array}
[c]{cl}%
c^{-1}\exp\left(  1/\left(  x^{2}+y^{2}-1\right)  \right)  & \text{if }%
x^{2}+y^{2}<1\\
0 & \text{if }x^{2}+y^{2}\geq1
\end{array}
\right.
\end{align*}
and
\[
c=\iint\nolimits_{x^{2}+y^{2}<1}\exp\left(  1/\left(  x^{2}+y^{2}-1\right)
\right)  dxdy.
\]

For details, see Choudary and Niculescu \cite{CN2014}, Theorem 11.7.6 (c), p. 404.

Now the proof can be completed by using an identity proved by W. H. Young in
\cite{Y1917}, p. 38 (and rediscovered later by Pe\v{c}ari\'{c} \cite{P1987}).
We provide here a shorter argument for Young's identity, based on repeated use
of integration by parts for absolutely continuous functions. Details
concerning the formula of integration by parts can be found in the monograph
of Hewitt and Stromberg \cite{HS}, Theorem 18.19, p. 287.

Using the aforementioned formula of integration by parts and Fubini's theorem
we have
\begin{multline}
\int_{a}^{b}\int_{c}^{d}f(x,y)w(x,y)dydx\tag{$Y1$}\label{Y1}\\
=\int_{a}^{b}\left[  f(x,d)\int_{c}^{d}w(x,y)dy-\int_{c}^{d}\left(  \int
_{c}^{y}w(x,t)dt\right)  \frac{\partial f(x,y)}{\partial y}dy\right]
dx\nonumber\\
=\int_{a}^{b}\left(  f(x,d)\int_{c}^{d}w(x,y)dy\right)  dx-\int_{a}^{b}\left[
\int_{c}^{d}\left(  \int_{c}^{y}w(x,t)dt\right)  \frac{\partial f(x,y)}%
{\partial y}dy\right]  dx\nonumber\\
=\int_{c}^{d}\left(  \int_{a}^{b}f(x,d)w(x,y)dx\right)  dy-\int_{c}^{d}\left[
\int_{a}^{b}\left(  \int_{c}^{y}w(x,t)dt\right)  \frac{\partial f(x,y)}%
{\partial y}dx\right]  dy\nonumber\\
=\int_{c}^{d}\left[  f(b,d)\int_{a}^{b}w(x,y)dx-\int_{a}^{b}\frac{\partial
f(x,d)}{\partial x}\left(  \int_{a}^{x}w(s,y)ds\right)  dx\right]
dy\nonumber\\
-\int_{c}^{d}\left[  \left(  \int_{a}^{b}\int_{c}^{y}w(x,t)dt\right)
\frac{\partial f(b,y)}{\partial y}\right]  dy\nonumber\\
+\int_{c}^{d}\int_{a}^{b}\left[  \left(  \int_{a}^{x}\int_{c}^{y}%
w(s,t)dtds\right)  \frac{\partial^{2}f(x,y)}{\partial x\partial y}dx\right]
dy\nonumber\\
=f(b,d)W(b,d)-\int_{a}^{b}\frac{\partial f(x,d)}{\partial x}%
W(x,d)dx\nonumber\\
-\int_{c}^{d}W(b,y)\frac{\partial f(b,y)}{\partial y}dy+\int_{a}^{b}\int
_{c}^{d}W(x,y)\frac{\partial^{2}f(x,y)}{\partial x\partial y}dydx.\nonumber
\end{multline}
Now the conclusion follows easily by taking into account the hypotheses
$(i^{\prime})$ and $(ii)$.
\end{proof}

\begin{remark}
\label{rem Abel1}An inspection of the argument of Theorem
\emph{\ref{thmAbelintegral}} shows that the conclusion still works if the
hypothesis $(i)$ is replaced by the following one:

$(i^{\prime})$ $f$ is nonnegative, continuous, $2d$-monotone decreasing and
admits a representation of the form
\begin{equation}
f(x,y)=f(a,c)+\int_{a}^{x}g_{1}(s)ds+\int_{c}^{y}g_{2}(t)dt+\int_{a}^{x}%
\int_{c}^{y}g(s,t)dtds, \tag{$AC$}\label{AC}%
\end{equation}
for suitable $g_{1}\in L^{1}\left(  [a,b]\right)  ,$ $g_{2}\in L^{1}\left(
[c,d]\right)  $ and $g\in L^{1}\left(  [a,b]\times\lbrack c,d]\right)  .$

The existence of the representation $(AC)$ is equivalent to the condition of
absolute continuity in the sense of Carath\'{e}odory. See \v{S}remr
\emph{\cite{S2010}}.
\end{remark}

Using a similar idea, one can prove the following companion of Theorem
\ref{thmAbelintegral}:

\begin{theorem}
\label{thmAbel2}Suppose that $f$ and $w$ are two real-valued functions defined
on $[a,b]\times\lbrack c,d]$ that fulfill the following conditions:

$(i)$ $f$ is nonnegative, continuously differentiable and $2d$-monotone
increasing$;$

$(ii)$ $w$ is integrable and $\widetilde{W}(x,y)=\int_{x}^{b}\int_{y}%
^{d}w(s,t)dsdt\geq0$ for all $\left(  x,y\right)  \in\lbrack a,b]\times\lbrack
c,d].$

Then%
\[
\int_{a}^{b}\int_{c}^{d}f(x,y)w(x,y)dxdy\geq f(a,c)\widetilde{W}(a,c).
\]

\end{theorem}

As above, the continuous differentiability of $f$ can be relaxed to absolute continuity.

The proof of Theorem \ref{thmAbel2} parallels that of Theorem
\ref{thmAbelintegral}, using the following variant of Young's identity:%
\begin{multline}
\int_{a}^{b}\int_{c}^{d}f(x,y)w(x,y)dxdy=f(a,c)\widetilde{W}(a,c)+\int_{a}%
^{b}\widetilde{W}(x,c)\frac{\partial f(x,c)}{\partial x}dx\nonumber\\
+\int_{c}^{d}\widetilde{W}(a,y)\frac{\partial f(a,y)}{\partial y}dy+\int
_{a}^{b}\int_{c}^{d}\widetilde{W}(x,y)\frac{\partial^{2}f(x,y)}{\partial
x\partial y}dxdy. \tag{$Y2$}\label{Y2}%
\end{multline}

A well known classical result concerning the Fourier coefficients of a convex
function of one variable asserts that%
\[
\int_{0}^{2\pi}f(x)\cos nxdx\geq0\text{\quad for all positive integers }n.
\]
See \cite{NP2006}, Exercise 7, p. 26. As a consequence,%
\[
\int_{0}^{2\pi}\int_{0}^{2\pi}f(x,y)\cos mx\cos nydxdy\geq0.
\]
for all convex functions $f:[0,2\pi]\times\lbrack0,2\pi]\rightarrow\mathbb{R}$
and all positive integers $m$ and $n.$ When cosines are replaced by sines, the
corresponding inequality may fail even in the one variable case. For example,%
\[
\int_{0}^{2\pi}x^{2}\sin nxdx<0\text{\quad for all positive integers }n.
\]
However, based on Theorem \ref{thmAbelintegral} and Theorem \ref{thmAbel2} we
will prove the following result:

\begin{proposition}
For all monotone and convex functions $f:[0,4\pi]\rightarrow\mathbb{R}$ and
all positive integers $m$ and $n,$%
\[
\int_{0}^{2\pi}\int_{0}^{2\pi}f(s+t)\sin ms\sin ntdsdt\geq0.
\]

\end{proposition}

\begin{proof}
Adding to $f$ a suitable constant, we may assume that $f\geq0.$ Then notice
that%
\[
W\left(  x,y\right)  =\int_{0}^{x}\int_{0}^{y}\sin ms\sin ntdsdt=\frac{\left(
1-\cos mx\right)  \left(  1-\cos ny\right)  }{mn}\geq0
\]
and%
\[
\tilde{W}\left(  x,y\right)  =\int_{x}^{2\pi}\int_{y}^{2\pi}\sin ms\sin
ntdsdt=\frac{\left(  1-\cos mx\right)  \left(  1-\cos ny\right)  }{mn}\geq0
\]
for all $x,y\in\lbrack0,2\pi].$ The conclusion of the corollary follows from
Theorem \ref{thmAbelintegral} when $f$ is decreasing, and from Theorem
\ref{thmAbel2} when $f$ is increasing.
\end{proof}

\begin{remark}
An inspection of the proof of Theorem \emph{\ref{thmAbelintegral}} easily
shows that the conclusion remains valid when the hypotheses $(i)\&(ii)$ are
replaced by the following ones:

$(i^{\prime\prime})$ $f$ is continuously differentiable, $2d$-alternating and
its partial derivatives $\frac{\partial f(x,d)}{\partial x}$ and
$\frac{\partial f(b,y)}{\partial y}$ are nonnegative$;$

$(ii^{\prime\prime})$ $w$ is integrable and $W(x,y)=\int_{a}^{x}\int_{c}%
^{y}w(s,t)dsdt\leq0$ for all $\left(  x,y\right)  \in\lbrack a,b]\times\lbrack
c,d].$

\smallskip

An illustration is offered by the functions $f(x,y)=\ln\left(  x^{2}%
+y^{2}\right)  $ and $w(x,y)=-\sin\left(  x+y\right)  $ on $(0,3\pi
/4]\times(0,3\pi/4]:$%
\[
\int_{0}^{3\pi/4}\int_{0}^{3\pi/4}\ln\left(  x^{2}+y^{2}\right)  \sin\left(
x+y\right)  dxdy\leq\ln\frac{9\pi^{2}}{8}\int_{0}^{3\pi/4}\int_{0}^{3\pi
/4}\sin\left(  s+t\right)  dsdt.
\]

A similar remark works in the case of Theorem \emph{\ref{thmAbel2}}.
\end{remark}

\section{An application to Riemann-Stieltjes integral of two variables}

The Riemann-Stieltjes integrals provide a unified approach to the theory of
random variables and proved useful in many fields like stochastic calculus and
statistical inference.

In analogy with integration over $\mathbb{R}$, the $2d$-monotone functions
$f:A\rightarrow\mathbb{R}$ give rise to Riemann-Stieltjes integrals. One first
define the integral of a characteristic function of a subinterval
$[a,b]\times\lbrack c,d]\subset A$ by the formula%
\[
\iint\nolimits_{\mathbb{R}^{2}}\chi_{\lbrack a,b]\times\lbrack c,d]}%
df(x,y)=f(a,c)-f(a,d)-f(b,c)+f(b,d),
\]
and then extend this formula by linearity and positivity to the linear space
$\mathcal{S}t\left(  A\right)  $ of step functions, that is, to the linear
combinations of characteristic functions of bounded intervals included in $A.$
Thus%
\[
\left\vert \iint\nolimits_{A}h(x,y)df(x,y)\right\vert \leq\lbrack
f;A]\cdot\left\Vert h\right\Vert _{\infty}\text{ }%
\]
for every $h\in\mathcal{S}t\left(  A\right)  .$ Since the elements of
$C_{c}(A)$ (that is, the space of all real-valued continuous functions with
compact support included in $A$) are uniform limits of step functions, one can
easily show that $df(x,y)$ is actually a positive Radon measure on $A.$

Under certain circumstances, the Riemann-Stieltjes integral can be reduced to
Riemann integral. For example, when $f$ is of class $C^{1}$ and admits
continuous mixed derivatives $\frac{\partial^{2}f}{\partial x\partial y}%
=\frac{\partial^{2}f}{\partial y\partial x},$ then%
\[
\iint\nolimits_{A}h(x,y)df(x,y)=\iint\nolimits_{A}h(x,y)\frac{\partial^{2}%
f}{\partial x\partial y}dxdy
\]
for all $h$ in $C_{c}(A).$ The technique of Dirac sequences, makes possible to
approximate any Stieltjes integral by Riemann integrals of this form.
Precisely,
\[
\iint\nolimits_{A}h(x,y)df(x,y)=\lim_{n\rightarrow\infty}\iint\nolimits_{A}%
h(x,y)\frac{\partial^{2}}{\partial x\partial y}\left(  \rho_{n}\ast f\right)
dxdy,
\]
where $\rho\in C_{c}(\mathbb{R}^{2})$ is any nonnegative function such that
$\iint\nolimits_{\mathbb{R}^{2}}\rho(x,y)dxdy=1$ and $\rho_{n}(x,y)=n^{2}%
\rho(nx,ny)$ for $n\geq1.$ See Willem \cite{W2003}, Th\'{e}or\`{e}me 11.14, p.
57, for details.

Taking into account the above discussion, one can restate Theorem
\ref{thmAbelintegral} as a formula of integration by parts for the
Riemann-Stieltjes integral:

\begin{theorem}
\label{thm Abel-Stieltjes}Suppose that $f$ and $g$ are two real-valued
functions defined on $[a,b]\times\lbrack c,d]$ that fulfill the following conditions:

$(i)$ $f$ is continuously differentiable and $2d$-monotone$;$

$(ii)$ $g$ is absolutely continuous in the sense of Carath\'{e}odory.

Then $f$ is integrable with respect to $g$ and
\begin{multline*}
\int_{a}^{b}\int_{c}^{d}f(x,y)dg(x,y)=f(b,d)g(b,d)\\
-\int_{a}^{b}\frac{\partial f(x,d)}{\partial x}g(x,d)dx-\int_{c}^{d}%
\frac{\partial f(b,y)}{\partial y}g(b,y)dy+\int_{a}^{b}\int_{c}^{d}%
g(x,y)df(x,y).
\end{multline*}

\end{theorem}

For $g(x,y)=\left(  x-\lfloor x\rfloor+1\right)  \left(  y-\lfloor
y\rfloor+1\right)  ,$ Theorem \ref{thm Abel-Stieltjes} yields a formula
relating double sums and double integrals

\begin{corollary}
\label{cor Abel-Stieltjes dis}If $f(x,y)$ is a continuously differentiable
function that admits a continuous second order partial derivative$\frac
{\partial^{2}f}{\partial x\partial y}$ in the rectangle $\left[  a,b\right]
\times\lbrack c,d]$, where $a,b,c,d\in\mathbb{Z},$ then%
\begin{multline*}
\sum_{a<m\leq b}\sum_{c<n\leq d}f(m,n)=\int_{a}^{b}\int_{c}^{d}f(x,y)dxdy+\int
_{a}^{b}\int_{c}^{d}\frac{\partial f(x,y)}{\partial x}\left(  x-\lfloor
x\rfloor\right)  dxdy\\
+\int_{a}^{b}\int_{c}^{d}\frac{\partial f(x,y)}{\partial y}\left(  y-\lfloor
y\rfloor\right)  dxdy+\int_{a}^{b}\int_{c}^{d}\frac{\partial^{2}%
f(x,y)}{\partial x\partial y}\left(  x-\lfloor x\rfloor\right)  \left(
y-\lfloor y\rfloor\right)  dxdy.
\end{multline*}

\end{corollary}

Here $\lfloor t\rfloor$ denotes the largest integer not greater than $t.$

A slightly more restrictive version of Corollary \ref{cor Abel-Stieltjes dis}
was first noticed by Rane (see \cite{R2003}, Corollary 2), who derived it from
the Euler-Maclaurin formula.


\begin{thebibliography}{99}                                                                                               %


\bibitem {Ab1826}N. H. Abel, Untersuchungen \"{u}ber die Reihe\textit{
}$1+\frac{m}{1}x+\frac{m(m-1)}{1\cdot2}x^{2}+\frac{m(m-1)(m-2)}{1\cdot2\cdot
3}x^{3}+\cdots,$ \emph{J. reine angew. Math.} \textbf{1} (1826), 311--339. See
also \emph{\OE uvres compl\`{e}tes de N. H. Abel}, t. \textbf{I}, pp. 66--92,
Christiania, 1839; available online at http://archive.org/stream/oeuvrescomplte01abel

\bibitem {AM}A. Aksoy and M. Martelli, Mixed partial derivatives and Fubini's
theorem, \emph{College Math. J}. \textbf{33} (2002) 126--130.

\bibitem {BB1986}D. Borwein and J. M. Borwein, A Note on Alternating Series in
Several Dimensions, \emph{Am. Math. Mon.} \textbf{93} (1986), No. 7, 531-539.

\bibitem {CN2014}A. D. R. Choudary and C. P. Niculescu, \emph{Real Analysis on
Intervals}, Springer, 2014.

\bibitem {H1904}G. H. Hardy, On the Convergence of Certain Multiple Series.
\emph{Proc. London Math. Soc.} \textbf{1} (1904), 124-128.

\bibitem {HS}E. Hewitt and K. Stromberg, \emph{Real and Abstract Analysis}.
Second printing corrected, Springer-Verlag, Berlin Heidelberg New York, 1969.

\bibitem {Ne2013}R. B. Nelsen, \emph{An Introduction to Copulas}, 2nd ed.,
Springer, 2006.

\bibitem {NP2006}C. P. Niculescu and L.-E. Persson, \emph{Convex Functions and
their Applications. A Contemporary Approach, }CMS Books in Mathematics
\textbf{23}, Springer-Verlag, New York, 2006.

\bibitem {NS2017}C. P. Niculescu and M. M. St\u{a}nescu, The
Steffensen-Popoviciu measures in the context of quasiconvex functions\emph{,
J. Math. Inequal.} \textbf{11} (2017), No. 2, 469--483.

\bibitem {P1987}J. Pe\v{c}ari\'{c}, Some Further Remarks on the Ostrowski
Generalization of \v{C}eby\v{s}ev's Inequality, \emph{J. Math. Anal. Appl.}
\textbf{123 }(1987), 18-33.

\bibitem {R2003}V. V. Rane, Analogues of Euler and Poisson summation formulae,
\emph{Proc. Indian Acad. Sci. (Math. Sci.} \textbf{113} (2003), No. 3, 213--221.

\bibitem {Sk1997}A. Sklar, Random variables, distribution functions, and
copulas -- a personal look backward and forward. In: L. R\"{u}schendorf, B.
Schweizer and M. Taylor (eds.), \emph{Distributions With Fixed Marginals \&
Related Topics}, IMS Lecture Notes -- Monograph Series No. \textbf{28}, 1996.

\bibitem {S2010}J. \v{S}remr, Absolutely continuous functions of two variables
in the sense of Carath\'{e}odory, \emph{Electronic Journal of Differential
Equations} (EJDE) \textbf{2010} (2010), Paper no. 154, 1--11.

\bibitem {St1919}J. F. Steffensen, On certain inequalities and methods of
approximation, \emph{J. Inst. Actuaries} \textbf{51 }(1919), 274-297.

\bibitem {W2003}M. Willem, \emph{Analyse fonctionnelle \'{e}l\'{e}mentaire},
Cassini, Paris, 2003.

\bibitem {Y1917}W. H. Young, On multiple integrals, \emph{Proc. Roy. Soc.
Series A}, \textbf{93} (1917), No. 647, pp. 28-41.
\end{thebibliography}
\end{document}